\documentclass{article}
\usepackage[dvips]{graphicx}
\usepackage{a4wide}%,times}

\usepackage[dvips]{graphics,color}      % usual driver
\usepackage{verbatim}
\usepackage{amsmath}
\usepackage{xspace}
\usepackage{pifont}
\usepackage{graphicx}
\usepackage{amssymb}
\usepackage{epic, eepic}
\usepackage{dsfont}
\usepackage{amssymb}
\usepackage{makeidx}

\usepackage{amsmath,afterpage}
\usepackage{epsf}

\def\0{\mathbf{0}}

\def\eps{\varepsilon}

\def\rr{\rightarrow}
\def\dr{\downarrow}

\def\beqa{\begin{eqnarray}}
\def\eeqa{\end{eqnarray}}
\def\beqas{\begin{eqnarray*}}
\def\eeqas{\end{eqnarray*}}

\newtheorem{theorem}{Theorem}[section]
\newtheorem{lemma}[theorem]{Lemma}
\newtheorem{proposition}[theorem]{Proposition}

\newtheorem{corollary}[theorem]{Corollary}

\newtheorem{definition}[theorem]{Definition}

\numberwithin{equation}{section}
\newcommand{\old}[1]{{}}

\def\endpf{{\ \hfill\hbox{\vrule width1.0ex height1.0ex}\parfillskip 0pt}}
\newenvironment{proof}{\noindent{\bf Proof:}}{\endpf}

\addtocounter{footnote}{0}

\newcommand{\qed}{\hfill\rule{2mm}{2mm}}

\newcommand{\bd}{\begin{displaymath}}
\newcommand{\ed}{\end{displaymath}}
\newcommand{\be}{\begin{equation}}
\newcommand{\ee}{\end{equation}}
\newcommand{\bq}{\begin{eqnarray}}
\newcommand{\eq}{\end{eqnarray}}
\newcommand{\bn}{\begin{eqnarray*}}
\newcommand{\en}{\end{eqnarray*}}
\newcommand{\dl}{\delta}

\title{Sample Path Properties of Volterra Processes}
\author{Leonid Mytnik\thanks{Partly supported by a grant from the  Israel Science Foundation.} \qquad \qquad \qquad Eyal Neuman$\mbox{}^*$ \\ \\ Faculty of Industrial Engineering \\ and Management  \\ Technion - Institute of Technology \\ Haifa 3200 \\ Israel }
\date{}

\begin{document}

\maketitle

\paragraph{Abstract.}
We consider the regularity of sample paths of Volterra processes.
These processes are defined as stochastic integrals

$$ M(t)=\int_{0}^{t}F(t,r)dX(r), \   \ t \in \mathds{R}_{+}, $$
where $X$ is a semimartingale and $F$ is a deterministic real-valued function. We derive the information on the modulus of continuity for these processes under regularity assumptions on the function $F$ and show that $M(t)$ has ``worst'' regularity properties at times of jumps of $X(t)$. We apply our results to obtain the optimal H\"older exponent for fractional L\'{e}vy processes.

\section{Introduction and main results}
\subsection{Volterra Processes} \label{Semimartingales}
A Volterra process is a process given by
\be \label{volterra}
 M(t)=\int_{0}^{t}F(t,r)dX(r), \   \ t \in \mathds{R}_{+},
\ee
where $\{X(t)\}_{t\geq0}$ is a semimartingale and $F(t,r)$ is a bounded deterministic real-valued function of two variables which sometimes is called a kernel.
One of the questions addressed in the research of Volterra and related processes is studying their regularity
properties.
It is also the  main goal of this paper. Before we describe our results let us give a short introduction to this area.
First, let us note that one-dimensional fractional processes, which are the close relative of Volterra processes,
have been extensively studied in the literature. One-dimensional fractional processes are usually defined by
\begin{eqnarray}
\label{frac:1}
X(t)&=& \int_{-\infty}^{\infty} F(t,r) dL(r),
\end{eqnarray}
where $L(r)$ is some stochastic process and $F(t,r)$ is some specific kernel. For example in the case of $L(r)$
being a two-sided standard Brownian motion and $F(t,r)=\frac{1}{\Gamma(H+1/2)}\left((t-s)^{H-1/2}_+ -(-s)^{H-1/2}_+\right)$, $X$ is called fractional Brownian motion with Hurst index $H$ (see  e.g. Chapter 1.2 of~\cite{oks08} and Chapter 8.2 of~\cite{nualart}). It is also known that the fractional Brownian motion with Hurst index $H$ is H\"older
continuous with any exponent less than $H$ (see e.g.~\cite{mandel.vanness}). Another prominent example is the case of fractional $\alpha$-stable L\'{e}vy process which can be also defined via~(\ref{frac:1}) with $L(r)$ being two-sided $\alpha$-stable L\'{e}vy process and
\bd
F(t,r)=a\{(t-r)^d_{+}-(-r)^d_{+}\}+b\{(t-r)^d_{-}-(-r)^d_{-}\}.
\ed \\
Takashima in~\cite{A8} studied path properties of this process.
%Takashima in \cite{A8} also studied path properties of fractional stable processes:
%\bd
%X(t)=\int_{-\infty}^{\infty}f_t(r)dZ(r).
%\ed
%Here:
%and $Z(t)$ is a stable L\'{e}vy process with exponent $\alpha$.
Takashima set the following conditions on the parameters:
$1<\alpha<2$, $0<d<1-\alpha^{-1}$ and $-\infty<a,b<\infty$, $|a|+|b|\neq 0$. It is proved in \cite{A8} that $X$ is a
self-similar process. Denote the jumps of $L(t)$ by
$\Delta_L (t)$: $\Delta_L(t)=L(t)-L(t-)$, $-\infty<t<\infty$. It is also proved in \cite{A8} that:
\bd
\lim_{h\downarrow 0}(X(t+h)-X(t))h^{-d} = a \Delta_L(t), \  \ 0<t<1, \ \ P- \rm{a.s.},
\ed
\bd
\lim_{h\downarrow 0}(X(t)-X(t-h))h^{-d} = -b \Delta_L(t), \  \ 0<t<1, \ \ P- \rm{a.s.}
\ed
Note that in his proof Takashima strongly used the self-similarity of the process $X$. \\\\
Another well-studied process is the so-called
%Marquardt in \cite{A4} defined the
fractional L\'{e}vy process,
%as follows.
which again is defined via~(\ref{frac:1}) for a specific kernel $F(t,r)$ and $L(r)$ being a two-sided L\'evy process.
For example,  Marquardt in \cite{A4} defined it as follows.
\begin{definition}\label{2} (Definition 3.1 in \cite{A4}): Let $L=\{L(t)\}_{t \in \mathds{R}}$ be a two-sided L\'{e}vy process on $\mathds{R}$ with
 $E[L(1)]=0$, $E[L(1)^2 ]<\infty$ and without a Brownian component. Let $F(t,r)$ be the following kernel function:
\bd
F(t,r)=\frac{1}{\Gamma(d+1)}[(t-r)^d_{+}-(-r)^d_{+}].
\ed
For fractional integration parameter $0<d<0.5$ the stochastic process
\bd
M_d(t)=\int_{-\infty}^{\infty}F(t,r)dL_r, \   \ t \in \mathds{R},
\ed
is called a \textit{fractional L\'{e}vy process}.
\end{definition}
%By a two-sided L\'{e}vy process we mean that $\{L(t)\}_{t \in \mathds{R}}$ is constructed from
%two independent copies of one-sided L\'{e}vy processes $\{L_1(t)\}_{t \geq 0}$ and $\{L_2(t)\}_{t \geq 0}$ such that
%\be \label{twosided}
%L(t)=
%\left \{ \begin{array}{ll}
%L_1(t), \ \ \textrm{    if    } t \geq 0, & \\ \\
%L_2(-t), \  \  \textrm{    if    } t < 0. & \\
%\end{array} \right.
%\ee
%Some results regarding the sample paths of fractional L\'{e}vy processes involve the H\"{o}lder
%regularity of the sample paths.
As for the regularity properties of fractional L\'evy process $M_d$ defined above,
Marquardt in \cite{A4} used an isometry of %the process
$M_d$ and the Kolmogorov continuity
criterion in order to prove that the sample paths of
%fractional L\'{e}vy process
$M_d$
%(see Definition \ref{2})
are $P$-a.s. local H\"{o}lder continuous of any order $\beta < d$.
%where $0< d < 0.5$.
%\\\\
 Moreover she proved that for every modification of $M_d$ and for every $\beta>d$:
\bd
P(\{\omega \in \Omega:M_d(\cdot,\omega) \not \in C^\beta[a,b]\})>0,
\ed
where $C^\beta[a,b]$ is the space of H\"{o}lder continuous functions of index $\beta$ on $[a,b]$.
Note that in this paper we are going to improve the result of Marquardt and show that for $d\in(0,0.5)$ the sample paths
of $M_d$ are $P$-a.s. H\"{o}lder continuous of any order $\beta \leq d$.  \\\\
The regularity properties of the analogous multidimensional processes have been also studied.
For example, consider the process
\be \label{MRT}
\hat{M}(t)=\int_{\mathds{R}^m}F(t,r)L(dr), \   \ t \in \mathds{R}^N,
\ee
where $L(dr)$ is some random measure and $F$ is a real valued function of two variables. A number of important results have been derived recently by Ayache, Roueff and Xiao in \cite{aya1}, \cite{aya2},
on the regularity properties of $\hat{M}(t)$ for some particular choices of $F$ and $L$.
As for the earlier work on the subject we can refer to K\^{o}no and Maejima in \cite{kono} and \cite{A6}. Recently, the regularity of related fractional processes was studied by Maejima and Shieh in \cite{Maej-Shieh}. 
We should also mention the book of Samorodnitsky and Taqqu \cite{A3} and the work of Marcus and Rosi\'{n}sky in \cite{marcus} where the regularity properties of processes related to $\hat{M}(t)$ in (\ref{MRT}) were also studied.

\subsection{Functions of Smooth Variation} \label{smoothvar}
In this section we make our assumptions on the kernel function $F(s,r)$ in (\ref{volterra}).
First we introduce the following notation. Denote
\bd
f^{(n,m)}(s,r) \equiv \frac{\partial^{n+m} f(s,r) }{\partial s^{n}  \partial r^{m}}, \ \ \forall n,m=0,1,\ldots.
\ed
We also define the following sets in $\mathds{R}^2$:
\bd
E=\{(s,r):-\infty<r \leq s < \infty \},
\ed
\bd
\tilde{E}=\{(s,r):-\infty<r < s < \infty \}.
\ed
We denote by $K$ a compact set in $E$, $\tilde{E}$ or $\mathds{R}$, depending on the context.
We define the following spaces of functions that are essential for the definition of functions
of \emph{smooth variation} and \emph{regular variation}.
\begin{definition} \label{CE}
Let $\mathbb{C}_{+}^{(k)}(E)$ denote the space of functions $f$ from a domain $E$ in $\mathds{R}^2$ to $\mathds{R^1}$ satisfying  \\
1. $f$ is continuous on $E$;\\
2. $f$ has continuous partial derivatives of order $k$ on $\tilde{E}$. \\
3. $f$ is strictly positive on $\tilde{E}$.
\end{definition}
Note that functions of smooth variation of one variable have been studied extensively in the literature; \cite{reg} is the standard reference for these and related functions.
Here we generalize the definition of functions of smooth variation to functions on $\mathds{R}^2$.
\begin{definition} \label{smt.var2.0}
Let $\rho>0$. Let $f \in \mathbb{C}_{+}^{(2)}(E)$ satisfy, for every compact set $K \subset \mathds{R}$, \\
a)
\bd
\lim_{h\downarrow 0}\sup_{ t \in K}\bigg| \frac{h f^{(0,1)}(t,t-h)}{f(t,t-h)}+\rho\bigg|=0,
\ed
b)
\bd
\lim_{h\downarrow 0}\sup_{t \in K}\bigg| \frac{h f^{(1,0)}(t+h,t)}{f(t+h,t)}-\rho\bigg|=0,
\ed
c)
\bd
\lim_{h\downarrow 0}\sup_{t \in K}\bigg| \frac{h^2 f^{(1,1)}(t,t-h)}{f(t,t-h)}+\rho(\rho-1)\bigg|=0,
\ed
d)
\bd
\lim_{h\downarrow 0}\sup_{t \in K}\bigg| \frac{h^2 f^{(0,2)}(t,t-h)}{f(t,t-h)}-\rho(\rho-1)\bigg|=0.
\ed
Then $f$ is called a function of smooth variation of index $\rho$ at the diagonal and is denoted as
$f \in SR_\rho^2(0+)$.
\end{definition}
It is easy to check that $f \in SR_\rho^2(0+)$, for $\rho>0$ satisfies $f(t,t)=0$ for all $t$. 
The trivial example for a function of smooth variation $SR_\rho^2(0+)$ is $f(t,r)=(t-r)^{\rho}$.
Another example would be $f(t,r)=(t-r)^{\rho}|\log(t-r)|^{\eta}$ where $\eta\in \mathds{R}$.

\subsection{Main Results} \label{TheoremsSection}

\paragraph{\textbf{Convention:}}
From now on we consider a semimartingale $\{X(t)\}_{t\geq 0}$ such that $X(0)=0$ $P$-a.s. Without
loss of generality we assume further that $X(0-)=0$, $P$-a.s. \medskip  \\
In this section we present our main results.
The first theorem gives us information about the regularity of increments of the process $M$.
%This information is very
%precise in the case when the process $X$ is discontinuous.
\begin{theorem} \label{thm1}
Let $F(t,r)$ be a function of smooth variation of index $d\in (0,1)$ and let $\{X(t)\}_{t\geq 0}$ be a semimartingale.
Define
\begin{displaymath}
M(t) =  \int_{0}^{t}F(t,r)dX(r), \ \ t\geq 0.
\end{displaymath}
Then,
\begin{displaymath}
\lim_{h\downarrow 0}  \frac{M(s+h)-M(s)}{F(s+h,s)}=\Delta_{X}(s) , \ \ \forall s \in [0,1], \ \ P-\rm{a.s.},
\end{displaymath}
where $\Delta_{X}(s)=X(s)-X(s-)$.
\end{theorem}
Information about the regularity of the sample paths of $M$ given in the above theorem is very precise in the case when the process $X$ is discontinuous. In fact, it shows that at the point of jump $s$, the increment of the process
behaves like $F(s+h,s)\Delta_X(s)$. \\\\
In the next theorem we give a uniform in time bound on the increments of the process $M$.
\begin{theorem}\label{thm2}
Let $F(t,r)$ and $\{M(t)\}_{t\geq 0}$ be as in Theorem \ref{thm1}. Then
\begin{displaymath}
\lim_{h\downarrow 0} \ \  \sup_{0< s<t <1,\ \ |t-s|\leq h}\frac{|M(t)-M(s)|}{F(t,s)} =  \sup_{s\in [0,1]} |\Delta _{X}(s)|, \   \ P-\rm{a.s.}
\end{displaymath}
\end{theorem}
Our next result, which in fact is a corollary of the previous theorem, improves the result of Marquardt from \cite{A4}.
\begin{theorem} \label{fracM}
Let $d\in(0,0.5)$. The sample paths of $\{M_d(t)\}_{t\geq 0}$, a fractional L\'{e}vy process, are $P$-a.s. H\"{o}lder continuous of order $d$ at any point $t\in\mathds{R}$.
\end{theorem}
In Sections \ref{SecThm1},\ref{SecThm2} we prove Theorems \ref{thm1} , \ref{thm2}. In Section \ref{fracmsec} we prove Theorem \ref{fracM}.

\section{Proof of Theorem \ref{thm1}} \label{SecThm1}
The proof of Theorems \ref{thm1} and \ref{thm2} uses ideas of Takashima in \cite{A8}, but does not use the self-similarity assumed there. \\\\
The goal of this section is to prove Theorem \ref{thm1}. First we prove the integration by parts formula in Lemma \ref{IntParts}. Later, in Lemma \ref{Varchange}, we decompose the increment $M(t+h)-M(t)$ into two components and then we analyze the limiting behavior of each of the components. This allows us to prove Theorem \ref{thm1}. \\\\
In the following lemma we refer to functions in $\mathds{C}^{(1)}(E)$, which is the space of functions from Definition  \ref{CE}, without the condition that $f>0$ on $\tilde{E}$. It is easy to show that functions of smooth variation satisfy the assumptions of this lemma.
\begin{lemma} \label{IntParts}
Let $X$ be a semimartingale such that $X(0)=0$ a.s. Let $F(t,r)$ be a function in $\mathds{C}^{(1)}(E)$ satisfying $F(t,t)=0$ for all $t\in \mathds{R}$. Denote $f(t,r)\equiv F^{(0,1)}(t,r)$.
Then,
\bd
\int_{0}^{t} F(t,r)dX(r)=-\int_{0}^{t} f(t,r) X(r) dr, \ \ P-\rm{ a.s.}
\ed
\end{lemma}
\begin{proof}
Denote $F_t(r)=F(t,r)$. By Corollary 2 in Section 2.6 of \cite{pro} we have
\be \label{chain32}
\int_{0}^{t} F_t(r-)dX(r)= X(t) F_t(t)- \int_{0}^{t} X(r-)dF_t(r)-[X,F_t]_t.
\ee
By the hypothesis we get $X(t) F_t(t)=0$. Since $F_t(\cdot)$ has continuous derivative and therefore is of bounded variation, it is easy to check that $[X,F_t]_t=0$, $P$-a.s.
\end{proof} \\\\
\textbf{Convention and Notation}\\
In this section we use the notation $F(t,r)$ for a smoothly varying function of index $d$ (that is, $F \in SR^2_d(0+)$), where $d$ is some number in $(0,1)$. We denote by $f(t,r) \equiv F^{(0,1)}(t,r)$, a \emph{smooth derivative of index} $d-1$ and let $SD^2_{d-1}(0+)$ denote the set of smooth derivative functions of index $d-1$. \\\\
In the following lemma we present the decomposition of the increments of the process $Y(t)$ that will be the key for the proof of Theorem \ref{thm1}.
\begin{lemma} \label{Varchange}
Let
\bd
Y(t)= \int_{0}^{t} f(t,r) X(r)dr, \qquad  t \geq 0.
\ed
Then we have
\bd
Y(t+\dl)-Y(t)= J_{1}(t,v,\dl)+J_{2}(t,v,\dl), \qquad  \forall t \geq 0,\ \ \dl>0,
\ed
where
\be \label{J1dec}
J_{1}(t,v,\dl) = \dl \int_{0}^{1} f(t+\dl,t+\dl-\dl v) X(t+\dl-\dl v)dv
\ee
and
\be \label{J2dec}
J_{2}(t,v,\dl) = \dl\int_{0}^{t/\dl} [f(t+\dl,t-\dl v)-f(t,t-\dl v)] X(t-\dl v)dv.
\ee
\end{lemma}

\begin{proof}
For any $t  \in [0,1], \ \ \dl > 0$ we have
\begin{eqnarray} \label{chain11}
Y(t+\dl)-Y(t) &= &\int_{0}^{t+\dl} f(t+\dl,r) X(r)dr - \int_{0}^{t} f(t,r) X(r)dr \\
&=& \int_{t}^{t+\dl} f(t+\dl,r) X(r)dr + \int_{0}^{t}[f(t+\dl,r)-f(t,r)] X(r)dr. \nonumber
\end{eqnarray}\\
By making a change of variables we are done.
\end{proof}\\\\
The next propositions are crucial for analyzing the behavior of $J_1$ and $J_2$ from the above lemma. 
\begin{proposition} \label{gUnifConv}
Let $f(t,r) \in SD^2_{d-1}(0+) $ where $d\in(0,1)$. Let $X(r)$ be a semimartingale.
 Denote
\bd
g_{\dl}(t,v) = \frac{f(t+\dl,t-\dl v)-f(t,t-\dl v)}{f(t+\dl,t)}, \ \ t \in[0,1], \ \ v \geq 0,  \ \ \dl>0.
\ed
Then
\bd
\lim_{\dl\downarrow 0}\bigg|  \int_{0}^{t/ \dl} g_{\dl}(t,v)X(t-\dl v)dv + \frac{1}{d} X(t-)\bigg|=0, \ \ \forall t \in [0,1], \ \ P-\rm{a.s.}
\ed
\end{proposition}

\begin{proposition} \label{con1} Let $f(t,r) \in SD^2_{d-1}(0+) $ where $d\in(0,1)$. Let $X(r)$ be a semimartingale.
Denote
\bd
f_{\dl}(t,v)=\frac{f(t+\dl,t+\dl-\dl v)}{f(t+\dl,t)} , \ \ t \in[0,1], \ \ v \geq 0,  \ \ \dl>0.
\ed
Then
\begin{displaymath}
\lim_{\dl\downarrow 0} \bigg| \int_{0}^{1} f_{\dl}(t,v) X(t+\dl(1-v))dv-\frac{1}{d}X(t)\bigg| =0, \ \ \forall t \in [0,1], \ \ P-\rm{a.s.}
\end{displaymath}
\end{proposition}
We first give a proof of Theorem \ref{thm1} based on the above propositions and then get back to the proofs of the propositions.
\paragraph {Proof of Theorem \ref{thm1}:}
From Lemma \ref{IntParts} we have
\be \label{chain1}
M(t+\dl)-M(t) = -(Y(t+\dl)-Y(t)), \ \ P-\rm{a.s.}
\ee
where
\bd
Y(t)= \int_{0}^{t} f(t,r) X(r)dr.
\ed
By Lemma \ref{Varchange}, for every $t \geq 0, \ \ \dl > 0$, we have
\be \label{Ydecomp}
Y(t+\dl)-Y(t)= J_{1}(t,v,\dl)+J_{2}(t,v,\dl),
\ee
For the first integral we get
\begin{displaymath}
\frac{J_{1}(t,v,\dl)}{\dl f(t+\dl,t)} =  \int_{0}^{1} f_{\dl}(t,v) X(t+\dl(1-v))dv.
\end{displaymath} \\
Now we apply Proposition \ref{con1} to get
\be \label{ptr11}
\lim_{\dl\downarrow 0} \frac{J_{1}(t,v,\dl)}{\dl f(t+\dl,t)} =  \frac{1}{d}X(t), \ \ \forall t \in [0,1], \ \ P-\rm{a.s.}
\ee
For the second integral we get
\begin{equation}
\frac{J_{2}(t,v,\dl)}{\dl f(t+\dl,t)}= \int_{0}^{t/\dl} g_{\dl}(t,v) X(t-\dl v)dv.
\end{equation} \\
By Proposition \ref{gUnifConv} we get
\be \label{ptr22}
\lim_{\dl\downarrow 0} \frac{J_{2}(t,v,\dl)}{\dl f(t+\dl,t)}=-\frac{1}{d}X(t-), \ \ \forall t \in [0,1], \ \ P-\rm{a.s.}
\ee\\
Combining (\ref{ptr11}) and (\ref{ptr22}) with (\ref{Ydecomp}) we get
\be \label{Meqn}
\lim_{\dl\downarrow 0}\frac{Y(t+\dl)-Y(t)}{\dl f(t+\dl,t)}= \frac{1}{d}\Delta_{X}(t), \ \ \forall t \in [0,1], \ \ P-\rm{a.s.}
\ee
Recall that $F^{(0,1)}(t,r)=f(t,r)$ where by our assumptions $F \in SR_d^2(0+)$. It is trivial to verify that
\be \label{limsmooth0}
 \lim_{h\downarrow 0}\sup_{ t \in [0,1]}\bigg| \frac{F(t,t-h)}{h f(t,t-h)}+\frac{1}{d}\bigg|=0.
\ee
Then by (\ref{chain1}), (\ref{Meqn}) and (\ref{limsmooth0}) we get
\be
\lim_{\dl\downarrow 0}  \frac{M(t+\dl)-M(t)}{F(t+\dl,t)}=\Delta_{X}(t) , \ \ \forall t \in [0,1], \ \ P-\rm{a.s.}
\ee
\qed \\\\
Now we are going to prove Propositions \ref{gUnifConv} and \ref{con1}.
First let us state two lemmas which are dealing with the properties of functions $f_{\dl}$ and $g_{\dl}$.
We omit the proofs as they are pretty much straightforward consequences of the properties of smoothly varying functions.
%The following lemmas are trivial results of the properties of smoothly varying functions.

\begin{lemma} \label{seq56tag2}
Let $f(t,r) \in SD^2_{d-1}(0+) $ where $d\in(0,1)$. Let $g_{\dl}(t,v)$ be defined as in  Proposition~\ref{gUnifConv}.
Then for every $h_0\in (0,1]$
\begin{itemize}
\item[{\bf (a)}]
\bd
 \lim_{\dl\downarrow 0}\sup_{h_0 \leq t \leq 1}\int_{h_0/\dl}^{t/ \dl} |g_{\dl}(t,v)|dv=0;
\ed
\item[{\bf (b)}]
\bd
\lim_{\dl\downarrow 0}\sup_{0 \leq t \leq 1} \bigg| \int_{0}^{h_0/ \dl} g_{\dl}(t,v)dv + \frac{1}{d}\bigg|=0;
\ed
\item[{\bf (c)}]
\bd
\lim_{\dl\downarrow 0}\sup_{h_0 \leq t \leq 1} \bigg| \int_{0}^{t/ \dl} g_{\dl}(t,v)dv + \frac{1}{d}\bigg|=0;
\ed
\item[{\bf (d)}]
\bd
\lim_{\dl\downarrow 0}\sup_{h_0 \leq t \leq 1}\bigg|\int_{0}^{h_0/ \dl} |g_{\dl}(t,v)|dv-\frac{1}{d}\bigg| =0.
\ed
\end{itemize}
\end{lemma}

%\begin{lemma} \label{seq56tagtag}
%Let $f(t,r) \in SD^2_{d-1}(0+) $ where $d\in(0,1)$. Let $g_{\dl}(t,v)$ be defined as in Proposition \ref{gUnifConv}.
%Then for every $h_0 \in (0,1]$
%\bd
%\lim_{\dl\downarrow 0}\sup_{0 \leq t \leq 1} \bigg| \int_{0}^{h_0/ \dl} g_{\dl}(t,v)dv + \frac{1}{d}\bigg|=0.
%\ed
%\end{lemma}

%\begin{corollary} \label{seq56tag}
%Let $f(t,r) \in SD^2_{d-1}(0+) $ where $d\in(0,1)$. Let $g_{\dl}(t,v)$ be defined as in Proposition \ref{gUnifConv}.
%Then for any $h_0\in(0,1)$,
%\bd
%\lim_{\dl\downarrow 0}\sup_{h_0 \leq t \leq 1} \bigg| \int_{0}^{t/ \dl} g_{\dl}(t,v)dv + \frac{1}{d}\bigg|=0.
%\ed
%\end{corollary}

%The following corollary follows immediately from Lemma \ref{seq56tag2}(a),(b).
% and \ref{seq56tagtag}.

%\begin{corollary} \label{seq56tag3}
%Let $f(t,r) \in SD^2_{d-1}(0+) $ where $d\in(0,1)$. Let $g_{\dl}(t,v)$ be defined as in Lemma \ref{gUnifConv}.
%Then for every $h_0 \in (0,1]$
%\bd
%\lim_{\dl\downarrow 0}\sup_{h_0 \leq t \leq 1}\bigg|\int_{0}^{h_0/ \dl} |g_{\dl}(t,v)|dv-\frac{1}{d}\bigg| =0.
%\ed
%\end{corollary}

%One may follow the same lines as in Lemma \ref{seq56tag2}(a),(b)
%--\ref{seq56tagtag}
%and
%Proposition \ref{seq56tag3}
%Lemma~\ref{seq56tag2}(d)???\footnote{Check} to get the following Corollary.

\begin{lemma} \label{conabsf} Let $f(t,r) \in SD^2_{d-1}(0+) $ where $d\in(0,1)$. Let $f_{\dl}(t,v)$ be
defined as in Proposition \ref{con1}.
Then,
\begin{itemize}
\item[{\bf (a)}]
\begin{displaymath}
\lim_{\dl\downarrow 0}\sup_{0 \leq t \leq 1} \bigg| \int_{0}^{1}|f_{\dl}(t,v) |dv - \frac{1}{d} \bigg|=0;
\end{displaymath}
 \item[{\bf (b)}]
\begin{displaymath}
\lim_{\dl\downarrow 0}\sup_{0 \leq t \leq 1} \bigg| \int_{0}^{1} f_{\dl}(t,v)dv-\frac{1}{d}\bigg| =0.
\end{displaymath}
\end{itemize}
\end{lemma}

%\begin{corollary} \label{fcon1} Let $f_{\dl}(s,v)$ be defined as in Proposition \ref{con1}.
%Then,
%\begin{displaymath}
%\lim_{\dl\downarrow 0}\sup_{0 \leq t \leq 1} \bigg| \int_{0}^{1} f_{\dl}(t,v)dv-\frac{1}{d}\bigg| =0.
%\end{displaymath}
%\end{corollary}

Now we will use Lemmas~\ref{seq56tag2}, \ref{conabsf} to
%finish the proof of
prove Propositions~\ref{gUnifConv}, \ref{con1}. At this point we also need to introduce the notation for the supremum norm on c\`{a}dl\`{a}g functions on $[0,1]$:
\bd
\|f\|_{\infty} = \sup_{0\leq t \leq 1} |f(t)|, \ \ f \in D_R[0,1],
\ed
where $D_R[0,1]$ is the class of real valued c\`{a}dl\`{a}g functions on $[0,1]$.
Since $X$ is a c\`{a}dl\`{a}g process we have
\be \label{xsup}
\|X\|_{\infty} < \infty, \ \ P-\rm{a.s.}
\ee
Note that for every $I\subset \mathds{R}$, $D_R(I)$ will denotes the class of real-valued c\`{a}dl\`{a}g functions on $I$.

\paragraph{Proof of Proposition \ref{gUnifConv}:}
Let us consider the following decomposition
\begin{eqnarray} \label{gendec}
\int_{0}^{t/ \dl} g_{\dl}(t,v)X(t-\dl v)dv + \frac{1}{d} X(t-) &=&
\int_{0}^{t/ \dl} g_{\dl}(t,v)[X(t-\dl v)- X(t-)]dv \\
&&{}+ X(t-) \bigg(\int_{0}^{t/ \dl} g_{\dl}(t,v)dv+ \frac{1}{d} \bigg) \nonumber \\ \nonumber \\
& =: & J_1(\dl,t) + J_2(\dl,t).\nonumber
\end{eqnarray}
By (\ref{xsup}) and
%Corollary \ref{seq56tag}
Lemma~\ref{seq56tag2}(c)
 we immediately get that for any arbitrarily small $h_0>0$, we have
\bd
\lim_{\dl\downarrow 0}\sup_{h_0 \leq t \leq 1}|J_2(\dl,t)| = 0, \ \ P-\rm{a.s.}
\ed
Since $h_0$ was arbitrary and $X(0-)=X(0)=0$, we get
\bd
\lim_{\dl\downarrow 0}|J_2(\dl,t)| = 0 , \ \  \forall t \in [0,1], \ \ P-\rm{a.s.}
\ed
Now to finish the proof it is enough to show that, $P-\rm{a.s.}$, for every $t \in [0,1]$
\be \label{J1LIM}
\lim_{\dl\downarrow 0}|J_1(\dl,t)| = 0.
\ee
For any $h_0 \in [0,t]$ we can decompose $J_1$ as follows
\begin{eqnarray} \label{decJ1}
J_1(\dl,t)&=& \int_{0}^{h_0/ \dl} g_{\dl}(t,v)[X(t-\dl v)- X(t-)]dv+\int_{h_0/ \dl}^{t/ \dl} g_{\dl}(t,v)[X(t-\dl v)- X(t-)]dv \nonumber \\
&=:& J_{1,1}(\dl,t)+J_{1,2}(\dl,t).
\end{eqnarray}
Let $\eps>0$ be arbitrarily small. $X$ is a c\`{a}dl\`{a}g process therefore, $P-\rm{a.s.}$ $\omega$, for every $t\in[0,1]$ we can fix $h_0\in[0,t]$ small enough such that
\be \label{epsbound11}
|X(t-\dl v,\omega) - X(t-,\omega)|<\eps, \ \ \textrm{for all } v \in (0,h_0 / \dl].
\ee
Let us choose such $h_0$ for the decomposition (\ref{decJ1}).
Then by (\ref{epsbound11}) and
%Corollary \ref{seq56tag3}
Lemma~ \ref{seq56tag2}(d) we can pick $\dl'>0$ such that for every $\dl\in(0,\dl')$ we have
\begin{eqnarray} \label{Jterm111}
|J_{1,1}(\dl,t)| &\leq& \frac{2\eps}{d}.
\end{eqnarray}
Now let us treat $J_{1,2}$. By (\ref{xsup}) and Lemma \ref{seq56tag2}(a) we get
\begin{eqnarray}  \label{Jterm2}
|J_{1,2}(\dl,t)| &\leq& 2\|X\|_{\infty}  \int_{h_0/ \dl}^{t/ \dl} |g_{\dl}(t,v)|dv \\
&\rr & 0, \ \ \rm{as} \ \ \dl \dr 0, \ \  \emph{P}-\rm{a.s.} \nonumber
\end{eqnarray}
Then by combining (\ref{Jterm111}) and (\ref{Jterm2}), we get (\ref{J1LIM})
and this completes the proof.
\qed

\paragraph{Proof of Proposition \ref{con1}:}
We consider the following decomposition
\begin{eqnarray*}
\int_{0}^{1} f_{\dl}(t,v)X(t+\dl(1-v))dv - \frac{1}{d} X(t) &=&
\int_{0}^{1} f_{\dl}(t,v)[X(t+\dl(1-v))- X(t)]dv \\
&&{}+ X(t) \bigg(\int_{0}^{1} f_{\dl}(t,v)dv- \frac{1}{d} \bigg) \\\\
& =: & J_1(\dl,t) + J_2(\dl,t).
\end{eqnarray*}
Now the proof follows along the same lines as that of Proposition \ref{gUnifConv}.
By (\ref{xsup}) and
%Corollary \ref{fcon1}
Lemma~\ref{conabsf}(b) we have
\bd
\lim_{\dl\downarrow 0}\sup_{0 \leq t \leq 1}|J_2(\dl,t)| = 0,  \ \ P-\rm{a.s.}
\ed
Hence to complete the proof it is enough to show that, $P-\rm{a.s.}$, for every $t \in [0,1]$
\bd
\lim_{\dl\downarrow 0}|J_1(\dl,t)| = 0.
\ed
Let $\eps>0$ be arbitrarily small. $X$ is a c\`{a}dl\`{a}g process; therefore, $P-\rm{a.s.}$ $\omega$, for every $t\in[0,1]$ we can fix $h_0$ small enough such that
\be \label{epsbound1}
|X(t+\dl(1-v),\omega) - X(t,\omega)|<\eps, \ \ \textrm{for all } v \in (0,h_0 / \dl].
\ee
Then by (\ref{epsbound1}) and
%Corollary \ref{conabsf}
Lemma~\ref{conabsf}(a) we easily get
\bd
\lim_{\dl\downarrow 0}|J_1(\dl,t)| = 0. \ \ \forall t \in [0,1], \ \ P-\rm{a.s.}
\ed
\qed

\section {Proof of Theorem \ref{thm2}}  \label{SecThm2}
Recall that by Lemma \ref{IntParts} we have
\be \label{chain2}
M(t)-M(s)=-(Y(t)-Y(s)), \ \  0 \leq s<  t,
\ee
where
\bd
Y(s)= \int_{0}^{s} f(s,r) X(r)dr, \ \ s>0.
\ed
Then by Lemma \ref{Varchange} we get:
\begin{eqnarray} \label{decY}
\frac{Y(s+\dl)-Y(s)}{\dl f(s+\dl,s)}&=& \frac{J_{1}(s,v,\dl)}{\dl f(s+\dl,s)}+\frac{J_{2}(s,v,\dl)}{\dl f(s+\dl,s)},
\ \ \dl>0.
\end{eqnarray}
Recall that $J_{1}$ and $J_{2}$ are defined in (\ref{J1dec}) and (\ref{J2dec}).
\paragraph{\textbf{Convention:}}
Denote by $\Gamma\subset\Omega$ the set of paths of $X(\cdot,\omega)$ which are right continuous and have left limit. By the assumptions of the theorem, $P(\Gamma)=1$. In what follows we are dealing with $\omega \in \Gamma$.
Therefore, for every $\varepsilon>0$ and $t>0$ there exists $\eta=\eta(\varepsilon,t,\omega)>0$ such that:
\begin{eqnarray} \label{jump}
|X(t-)-X(s)|&\leq& \varepsilon, \    \qquad      \textrm{for all}  \  \ \ s \in [t-\eta , t), \\
|X(t)-X(s)|&\leq& \varepsilon,      \    \qquad      \textrm{for all}  \  \ \ s \in [t,t+\eta]. \nonumber
\end{eqnarray}
Let us fix an arbitrary $\varepsilon>0$. The interval $[0,1]$ is compact; therefore there exist points $t_{1},\ldots, t_{m}$ that define a cover of $[0,1]$ as follows:
\begin{displaymath}
[0,1]\subset\bigcup_{k=1}^{m}\big(t_{k}-\frac{\eta_{k}}{2} , t_{k}+\frac{\eta_{k}}{2}\big),
\end{displaymath}
where we denote $\eta_{k}=\eta(\eps,t_k)$. Note that if $\Delta _{X}(s)>2\varepsilon$ then $s=t_{k}$ for some $k$. \\
We can also construct this cover in a way that
\be \label{tor}
\inf_{k\in \{2,\ldots,m\}} \big(t_k-\frac{\eta_k}{2}\big) \geq t_1.
\ee
Also since $X(t)$ is right continuous at $0$, we can choose $t_1$ sufficiently small such that
\be \label{Xor}
\sup_{t\in(0,t_1+\frac{\eta_1}{2})}|X(t)| \leq \eps.
\ee
Denote:
\be \label{bk}
B_{k}=\big(t_{k}-\eta_{k} , t_{k}+\eta_{k} \big), \ \ B_{k}^{*}=\big(t_{k}-\frac{\eta_{k}}{2} , t_{k}+\frac{\eta_{k}}{2}\big).
\ee
Note that the coverings $B_{k}$ and $B_{k}^{*}$ we built above are random---they depend on a particular realization of $X$. For the rest of this section we will be working with the particular realization of $X(\cdot,\omega)$, with $\omega \in \Gamma$ and the corresponding coverings $B_{k}$, $B_{k}^{*}$. All the constants that appear below may depend on $\omega$ and the inequalities should be understood $P$-a.s.\\\\
Let $s,t\in B_{k}^{*}$ and denote $\delta=t-s$.
Recall the notation from Propositions \ref{gUnifConv} and \ref{con1}.
Let us decompose $\frac{J_{i}(s,\dl)}{\dl f(s+\dl,s)}, \ \ i=1,2,$ as follows:
\begin{eqnarray*}
\frac{J_{1}(s,\dl)+J_{2}(s,\dl)}{\dl f(s+\dl,s)} &=& X(t_{k}-) \bigg[\int_{0}^{1} f_{\dl}(s,v)dv+\int_{0}^{s/ \dl} g_{\dl}(s,v)dv \bigg]  \\
&&{} + \Delta _{X}(t_{k}) \bigg[\int_{0}^{1} f_{\dl}(s,v) \mathds{1}_{\{s+\dl(1-v)\geq t_{k}\}}dv +\int_{0}^{s/ \dl} g_{\dl}(s,v) \mathds{1}_{\{s-\dl v\geq t_{k}\}}dv \bigg]  \\
&&{} +\int_{0}^{1} f_{\dl}(s,v) \mathds{1}_{\{s+\dl(1-v)< t_{k}\}} [X(s+\dl (1-v))- X(t_{k}-)]dv \\
&&{} + \int_{0}^{s/ \dl} g_{\dl}(s,v) \mathds{1}_{\{s-\dl v< t_{k}\}} [X(s-\dl v)- X(t_{k}-)]dv \\
&&{} + \int_{0}^{1} f_{\dl}(s,v) \mathds{1}_{\{s+\dl(1-v)\geq t_{k}\}}[X(s+\dl (1-v))- X(t_{k}+)]dv \\
&&{} +\int_{0}^{s/ \dl} g_{\dl}(s,v) \mathds{1}_{\{s-\dl v \geq t_{k}\}}[X(s-\dl v)- X(t_{k}+)]dv  \\ \\
&=:&  D_{1}(k,s,\dl) + D_{2}(k,s,\dl) +\ldots+D_{6}(k,s,\dl) ,
\end{eqnarray*} \\
where $\mathds{1}$ is the indicator function.
The proof of Theorem \ref{thm2} will follow as we handle the terms $D_i$, $i=1,2,\ldots,6$ via a series of lemmas.

\begin{lemma} \label{firstcor}
There exists a sufficiently small $h_{\ref{firstcor}}>0$ such that
\be \label{D1Bound}
|D_1(k,s,\dl)| \leq \bigg(|X(t_k-)|+\frac{4}{d}\bigg)\eps, \ \ \forall k \in \{1,\ldots,m\}, \ \ s\in B_k^{*}, \ \ \dl\in(0,h_{\ref{firstcor}}).
\ee
\end{lemma}

\begin{proof}
By
%Lemma \ref{seq56tag},
Lemma~\ref{seq56tag2}(c),
%Corollary \ref{fcon1}
Lemma~\ref{conabsf}(b) and by our assumptions on the covering we get (\ref{D1Bound}) for $k=2,\ldots,m$.
As for $k=1$, we get by (\ref{Xor})
\be \label{xt1}
|X(t_1-)| \leq \eps.
\ee
By
%Lemma \ref{seq56tagtag}
Lemma~\ref{seq56tag2}(b) we have
\be \label{t1eta}
\sup_{ 0 \leq s \leq 1}\bigg|\int_{0}^{(t_1+\eta_1)/ \dl} g_{\dl}(s,v)dv +\frac{1}{d}\bigg| <\varepsilon/2.
\ee
Hence by %Lemma \ref{seq56tag},
Lemma \ref{seq56tag2}(c),
 (\ref{xt1}) and (\ref{t1eta}), for a sufficiently small $\dl$, we get
\begin{eqnarray*}
\sup_{s\in B_1^{*}} |D_1(1,s,\dl)| & \leq & \eps \sup_{s\in B_1^{*}} \bigg\{ \bigg|\int_{0}^{1} f_{\dl}(s,v)dv \bigg| + \int_{0}^{(t_1+\eta_1)/ \dl} |g_{\dl}(s,v)|dv  \bigg\} \\
&\leq & \eps\frac{4}{d} ,
\end{eqnarray*}
and (\ref{D1Bound}) follows.
\end{proof}  \\\\
To handle the $D_2$ term we need the following lemma.

\begin{lemma} \label{con3}
Let $g_{\dl}(s,v)$ and $f_{\dl}(s,v)$ be defined as in Propositions \ref{gUnifConv} and \ref{con1}.
Then there exists $h_{\ref{con3}}>0$ such that for all $\dl  \in (0,h_{\ref{con3}})$,  \\
\bd
\bigg|\int_{0}^{1} f_{\dl}(s,v) \mathds{1}_{\{s+\dl(1-v)\geq t_{k}\}}dv +\int_{0}^{s/ \dl} g_{\dl}(s,v) \mathds{1}_{\{s-\dl v\geq t_{k}\}}dv \bigg| \leq \frac{1}{d}+\varepsilon, \ \ \forall k \geq 1,\textrm{ } s \in[0,1].
\ed
\end{lemma}

\begin{proof}
We introduce the following notation
\bd
I_{1}(s,\dl)= \int_{0}^{1} f_{\dl}(s,v) \mathds{1}_{\{s+\dl(1-v)\geq t_{k}\}}dv,
\ed
\bd
I_{2}(s,\dl)= \int_{0}^{s/ \dl} g_{\dl}(s,v) \mathds{1}_{\{s-\dl v\geq t_{k}\}}dv.
\ed
From Definition \ref{smt.var2.0}, it follows that there exists $h_1>0$, such that for every $\dl \in h_1$, $v \in (0,\frac{h_1}{2\dl})$ and $s\in[0,1]$
\begin{eqnarray} \label{gdlTerm}
g_{\dl}(s,v) & \leq & 0,
\end{eqnarray}
and
\begin{eqnarray} \label{fdlTerm}
f_{\dl}(s,v) & \geq & 0.
\end{eqnarray}
By Lemma \ref{seq56tag2}(a), we can fix a sufficiently small $h_2\in(0,h_1/2)$ such that for every $\dl \in (0,h_2)$, we have
\be \label{intTerm}
\int_{h_1/(2\dl)}^{s/ \dl} |g_{\dl}(s,v)|dv \leq \eps/2, \ \ \forall s \in [h_1/2,1],
\ee
where $\eps$ was fixed for building the covering $\{B_k^*\}_{k=1}^{m}$.
Then, by (\ref{intTerm}), we have
\begin{eqnarray} \label{Bund13}
|I_1(s,\dl)+I_2(s,\dl)| &\leq& \bigg|I_1 +  \int_{0}^{(s\wedge\frac{h_1}{2})\dl} g_{\dl}(s,v) \mathds{1}_{\{v \leq \frac{s-t_{k}}{\dl}\}}dv\bigg|+\eps/2,
\end{eqnarray}
for all $s \in [0,1], \ \ \dl \in (0,h_2)$. \\\\
By (\ref{fdlTerm}) and the choice of $h_2\in (0,\frac{h_1}{2}\big)$, we get
\begin{eqnarray}\label{help14}
I_1(s,\dl) &\geq& 0, \ \  \forall s \in [0,1], \ \ \dl \in (0,h_2).
\end{eqnarray}
By (\ref{gdlTerm}) we have
\begin{eqnarray} \label{help13}
\int_{0}^{(s\wedge\frac{h_1}{2})\dl} g_{\dl}(s,v) \mathds{1}_{\{v \leq \frac{s-t_{k}}{\dl}\}}dv
& \leq & 0, \ \  \forall s \in [0,1], \ \ \dl \in (0,h_2).
\end{eqnarray}
Then by (\ref{Bund13}), (\ref{help14}) and (\ref{help13}) we get
\begin{eqnarray} \label{Bund15}
|I_1(s,\dl)+I_2(s,\dl)|
&\leq& \max\bigg\{\int_{0}^{1} f_{\dl}(s,v)dv,\bigg|\int_{0}^{( s \wedge \frac{h_1}{2}) / \dl} g_{\dl}(s,v)dv \bigg| \bigg\} +\eps/2
\end{eqnarray}
for all $s \in [0,1], \ \ \dl \in (0,h_2)$. \\\\
By (\ref{Bund15}),
%Corollary \ref{fcon1}
Lemma~\ref{conabsf}(b) and
%Corollary \ref{seq56tag3}
Lemma~\ref{seq56tag2}(d) we can fix $h_{\ref{con3}}$ sufficiently small such that
\bd
|I_1(s,\dl)+I_2(s,\dl)| \leq \frac{1}{d}+\eps
\ed
and we are done.
\end{proof} \\\\
Note that
\be \label{d211}
|D_{2}(k,s,\dl)| = \Delta _{X}(t_{k}) \bigg[\int_{0}^{1} f_{\dl}(s,v) \mathds{1}_{\{s+\dl(1-v)\geq t_{k}\}}dv +\int_{0}^{s/ \dl} g_{\dl}(s,v) \mathds{1}_{\{s-\dl v\geq t_{k}\}}dv \bigg]
\ee
Then the immediate corollary of Lemma \ref{con3} and (\ref{d211}) is
\begin{corollary} \label{corrd2}
\bd
|D_{2}(k,s,\dl)|\leq |\Delta _{X}(t_{k})| \big(\varepsilon+\frac{1}{d}\big) , \ \ \forall k \in \{1,\ldots,m\}, \ \ s\in[0,1], \ \ \dl\in(0,h_{\ref{con3}}).
\ed
\end{corollary}
One can easily deduce the next corollary of Lemma \ref{con3}.
\begin{corollary} \label{rcorrd2}
There exists $h_{\ref{rcorrd2}}>0$ such that
\be \label{d2}
\big|\sup_{s\in B_k^{*}} |D_{2}(k,s,\dl)|-|\Delta _{X}(t_{k})|\frac{1}{d} \big| \leq \varepsilon|\Delta _{X}(t_{k})| , \ \ \forall k \in \{1,\ldots,m\},  \ \ \dl\in(0,h_{\ref{rcorrd2}}).
\ee
\end{corollary}
\begin{proof}
By Corollary \ref{corrd2} we have
\bd
\sup_{s\in B_k^{*}}|D_{2}(k,s,\dl)|\leq |\Delta _{X}(t_{k})|\frac{1}{d}+|\Delta _{X}(t_{k})|\varepsilon , \ \ \forall k \in \{1,\ldots,m\}, \ \ s\in[0,1], \ \ \dl\in(0,h_{\ref{con3}}).
\ed
To get (\ref{d2}) it is enough to find $s\in  B_k^{*}$ and $h_{\ref{rcorrd2}} \in (0,h_{\ref{con3}})$ such that for all $\dl\in(0,h_{\ref{rcorrd2}})$.
\be \label{d2low}
|D_{2}(k,s,\dl)|\geq |\Delta _{X}(t_{k})|\frac{1}{d}-|\Delta _{X}(t_{k})|\varepsilon , \ \ \forall k \in \{1,\ldots,m\}.
\ee
By picking $s=t_k$ we get
\bd
|D_{2}(k,t_k,\dl)|= \bigg|\Delta _{X}(t_{k}) \int_{0}^{1} f_{\dl}(t_k,v)dv \bigg|.
\ed
Then by
%Lemma \ref{fcon1},
Lemma~\ref{conabsf}(b), (\ref{d2low}) follows and we are done.
\end{proof} \\\\
The term $|D_{3}(k,s,\dl)|+|D_{5}(k,s,\dl)|$ is bounded by the following lemma.

\begin{lemma} \label{D3D5}
There exists a sufficiently small $h_{\ref{D3D5}}$ such that
\bd
|D_{3}(k,t_k,\dl)|+|D_{5}(k,t_k,\dl)|\leq \varepsilon \cdot \frac{4}{d},  \ \ \forall k \in \{1,\ldots,m\}, \ \ s \in B_{k}^{*}, \ \ \forall \dl \in(0,h_{\ref{D3D5}}).
\ed
\end{lemma}

\begin{proof}
By the construction of $B_k$ we get that
\be \label{D3}
|D_{3}(k,s,\dl)|\leq \varepsilon \int_{0}^{1}| f_{\dl}(s,v) |\mathds{1}_{\{s+\dl(1-v)< t_{k}\}}dv, \ \ \forall s \in B^*_k, \ \ k=\{1,\ldots,m\}, \ \ \dl \in (0,\eta_k/2),
\ee
and
\be \label{D5}
|D_{5}(k,s,\dl)|\leq \varepsilon \int_{0}^{1}| f_{\dl}(s,v) |\mathds{1}_{\{s+\dl(1-v)\geq t_{k}\}}dv, \ \ \forall s \in B^*_k, \ \ k=\{1,\ldots,m\}, \ \ \dl \in (0,\eta_k/2).
\ee
From (\ref{D3}), (\ref{D5}) we get
\bd
|D_{3}(k,s,\dl)|+|D_{5}(k,s,\dl)|\leq 2\varepsilon \int_{0}^{1}| f_{\dl}(s,v) |dv ,\ \ \forall \ k=\{1,\ldots,m\},  \ \ s \in B_k^*, \ \ \dl \in (0,\eta/2).
\ed \\
By
%Corollary \ref{conabsf}
Lemma~\ref{conabsf}(a) the result follows.
\end{proof} \\\\
$|D_4(k,s,\dl)|$ is bounded in the following lemma.

\begin{lemma}\label{con4}
There exists $h_{\ref{con4}}>0$  such that for all $\dl  \in (0,h_{\ref{con4}})$:  \\
\be \label{con4eq}
|D_4(k,s,\dl)| \leq \varepsilon \big(\frac{2}{d}+2\|X\|_{\infty}\big) ,\ \ \forall s  \in B^{*}_{k}, \ \ k \in \{1,\ldots,m\}.
\ee
\end{lemma}

\begin{proof}
Let $\eps>0$ be arbitrarily small and fix $k \in \{1,\ldots,m\}$.
First we consider the case $s-t_{k}>0, \ \  k \in \{1,\ldots,m\}$.
\begin{eqnarray} \label{Sstar}
\bigg|\int_{0}^{s/ \dl} g_{\dl}(s,v) \mathds{1}_{\{s-\dl v< t_{k}\}} [X(s-\dl v)- X(t_k-)]dv \bigg| \qquad\qquad \qquad  \\  \nonumber \\
 \qquad \leq \int_{(s-t_{k})/\dl}^{(s-t_{k})/\dl+\eta_k/(2\dl)} |g_{\dl}(s,v)||X(s-\dl v)- X(t_k-)|dv  \qquad \qquad \nonumber \\ \nonumber \\
 +\bigg|\int_{(s-t_{k})/\dl+\eta_k/(2\dl)}^{s/ \dl}\mathds{1}_{\{ t_k>\eta_k/2 \}} g_{\dl}(s,v)[X(s-\dl v)- X(t_k-)]dv \bigg|  \nonumber \\ \nonumber \\
:= |I_1(k,s,\dl)|+|I_2(k,s,\dl)|.    \qquad \qquad \qquad \qquad \qquad \qquad   \qquad \quad   \nonumber
\end{eqnarray}
Note that the indicator in $I_2(k,s,\dl)$ makes sure that $s/\dl > (s-t_{k})/\dl+\eta_k/(2\dl)$. \\\\
By the defintion of $B_k$ in (\ref{bk}) and by
%Corollary \ref{seq56tag3},
Lemma~ \ref{seq56tag2}(d)
there exists $h_1>0$ such that for every $\dl \in (0,h_1)$ we have uniformly on $s \in B^{*}_{k}\bigcap [t_k,1]$
\begin{eqnarray} \label{IS1}
|I_1(k,s,\dl)| &\leq& \frac{2}{d}\eps.
\end{eqnarray}
By Lemma \ref{seq56tag2}(a), there exists $h_2 \in (0,h_1)$ such that for every $\dl \in (0,h_2)$ we have uniformly on $s \in B^{*}_{k}\bigcap [t_k,1]$ (note that if $ t_k\leq \eta_k/2$ then $I_2(k,s,\dl)=0$)
\begin{eqnarray}\label{IS2}
|I_2(k,s,\dl)|
&\leq&  2\|X\|_{\infty} \eps.
\end{eqnarray}
By (\ref{Sstar}), (\ref{IS1}) and (\ref{IS2}) we get (\ref{con4eq}). \\\\
Consider the case $s \leq t_{k},  \ \ s\in B_k^*,\ \  k \in \{1,\ldots,m\}$. Then we have
\begin{eqnarray} \label{star2}
\int_{0}^{s/ \dl} g_{\dl}(s,v) \mathds{1}_{\{s-\dl v< t_{k}\}} [X(s-\dl v)- X(t_k-)]dv
&=& \int_{0}^{\eta_k/(2\dl)} g_{\dl}(s,v)[X(s-\dl v)- X(t_k-)]dv  \nonumber \\
&+& \int_{\eta_k/(2\dl)}^{s/\dl} g_{\dl}(s,v)[X(s-\dl v)- X(t_k-)]dv  \nonumber \\ \nonumber \\
&=:& J_1(k,s,\dl)+J_2(k,s,\dl).
\end{eqnarray}
Note that if $v\in(0,\eta_k/(2\dl))$, $s \leq t_{k}$ and $s\in B_k^*$, then $s-\dl v \in (t_k-\eta_k,t_k)$. Hence, by the construction of $B_k$ we have
\be \label{Xor1}
\sup_{v\in(0,\eta_k/(2\dl))}|X(s-\dl v)- X(t_k-)| \leq \eps, \ \ \forall s \leq t_{k}, \ \ s\in B_k^*.
\ee
By (\ref{Xor1}) and
%Corollary \ref{seq56tag3},
Lemma~ \ref{seq56tag2}(d),  there exists $h_4\in (0,h_3)$ such that
\begin{eqnarray}\label{star3}
|J_1(k,s,\dl)| &\leq& \eps \frac{2}{d}, \ \ \forall \dl \in (0,h_4), \ \ s \leq t_{k}, \ \ s \in B_k^*, \ \  k \in \{1,\ldots,m\}.
\end{eqnarray}
By Lemma \ref{seq56tag2}(a), exists $h_{\ref{con4}}\in (0,h_4)$ such that
\begin{eqnarray}\label{star4}
|J_2(k,s,\dl)| &\leq&  2\|X\|_{\infty} \eps , \ \ \forall \dl \in (0,h_{\ref{con4}}), \ \ s \leq t_{k}, \ \ s \in B_k^*, \ \  k \in \{1,\ldots,m\}.
\end{eqnarray}
By combining (\ref{star3}) and (\ref{star4}) with (\ref{star2}), the result follows.
\end{proof} \\\\
$|D_6(k,s,\dl)|$ is bounded in the following lemma.

\begin{lemma}\label{con5}
There exists $h_{\ref{con5}}>0$  such that for all $\dl  \in (0,h_{\ref{con5}})$:  \\
\bd
|D_6(k,s,\dl)| \leq \frac{2\eps}{d}  ,\ \ \forall s  \in B^{*}_{k}, \ \ k \in \{1,\ldots,m\}.
\ed
\end{lemma}

\begin{proof}
Recall that
\bd
B_{k}^{*}=\big(t_{k}-\frac{\eta_{k}}{2},t_{k}+\frac{\eta_{k}}{2}\big), \ \ k \in \{1,\ldots,m\}.
\ed
Note that
\be \label{D6MIN}
|D_6(k,s,\dl)|=0, \ \ \forall s \in \big(t_{k}-\frac{\eta_{k}}{2},t_k\big], \ \ k \in \{1,\ldots,m\}.
\ee
Hence we handle only the case of $s>t_k$, $s\in B_k^*$. One can easily see that in this case
\begin{eqnarray*}
D_{6}(k,s,\dl)&=& \int_{0}^{(s-t_k)/ \dl} g_{\dl}(s,v) [X(s-\dl v)- X(t_{k}+)]dv.
\end{eqnarray*}
Then by the construction of $B^*_k$, for every $s\in B^{*}_{k}$, $s>t_k$ we have
\be \label{epsbound44}
|X(s-\dl v) - X(t_k)|\leq \eps, \ \ \textrm{for all } v \in (0,(s-t_k)/ \dl].
\ee
We notice that if $s\in B_{k}^{*}$ and $s>t_k$ then $0<s-t_k<\eta_k/2$, for all $k=1,\ldots,m$.
Denote by $\eta=\max_{k=1,\ldots,m}\eta_k$.
Then by (\ref{epsbound44}) and
%Corollary \ref{seq56tag3}
Lemma~\ref{seq56tag2}(d) we can pick $h_{\ref{con5}}>0$ such that for every $\dl  \in (0,h_{\ref{con5}})$ we have
\begin{eqnarray} \label{Jterm1}
|D_{6}(k,s,\dl)| &\leq& \frac{2\eps}{d}, \ \ \forall s\in \big(t_k,t_{k}+\frac{\eta_{k}}{2}\big),\ \ k \in \{1,\ldots,m\}.
\end{eqnarray}
Then by (\ref{D6MIN}) and (\ref{Jterm1}) for all $\dl  \in (0,h_{\ref{con5}})$, the result follows.
\end{proof} \\\\
Now we are ready to complete the proof of Theorem \ref{thm2}. By Lemmas \ref{firstcor}, \ref{D3D5}, \ref{con4}, \ref{con5} and by Corollary \ref{rcorrd2}, there exists $h^*$ small enough and $C_{\ref{Cleb}}= 6\|X\|_\infty+ \frac{12}{d}$ such that
\be \label{Cleb}
\bigg| \sup_{s\in B_k^*} \frac{|J_{1}(s,\dl)+J_{2}(s,\dl)|}{|\dl f(s+\dl,s)|}- \frac{1}{d}|\Delta _{X}(t_{k})| \bigg| \leq\varepsilon \cdot C_{\ref{Cleb}} , \ \ \forall k \in \{1,\ldots,m\}, \ \ \dl \in (0,h^*), \ \ P-\rm{a.s.}
\ee \\
By (\ref{limsmooth0}), (\ref{chain2}) and (\ref{decY}) we can choose $h\in(0,h^*)$ to be small enough such that
\be \label{all}
 \bigg|\sup_{0 < s<t < 1, \ \ |t-s|\leq \dl,\ \ s \in B_{k}^{*}}\frac{|M(t)-M(s)|}{F(t,s)} - |\Delta _{X}(t_{k})| \bigg| \leq\varepsilon \cdot C_{\ref{all}},  \ \ \forall k \in \{1,\ldots,m\}, \ \ \dl \in (0,h), \ \ P-\rm{a.s.}
\ee
where $C_{\ref{all}}=C_{\ref{Cleb}}+1$.
By the construction of the covering $B_k$, for any point $s \not \in \{t_1,\ldots,t_k\}$, $|\Delta _{X}(s)|\leq 2\eps$. Set $C_{\ref{cnew}} = C_{\ref{all}} + 2$. Then, by (\ref{all}) we get
\begin{eqnarray} \label{cnew}
 \bigg|\sup_{0 < s<t < 1,\ \ |t-s|\leq \dl}\frac{|M(t)-M(s)|}{F(t,s)} - \sup_{s\in[0,1]}|\Delta _{X}(s)| \bigg| &\leq& \varepsilon \cdot C_{\ref{cnew}},  \ \ \forall k \in \{1,\ldots,m\}, \nonumber  \\
&&{} \dl \in (0,h), \ \ P-\rm{a.s.}
\end{eqnarray}
Since $C_{\ref{cnew}}$ is independent of $m$, and since $\eps$ was arbitrarily small the result follows.
\qed

\section{Proof of Theorem \ref{fracM} } \label{fracmsec}
In this section we prove Theorem \ref{fracM}. In order to prove Theorem \ref{fracM} we need the following lemma.
\begin{lemma}  \label{fracL}
Let $L(t)$ be a two-sided L\'{e}vy process with $E[L(1)]=0$, $E[L(1)^2 ]<\infty$ and without a Brownian component. Then for $P$-a.e. $\omega$, for any $t\in \mathds{R}$, $a \leq 0$ such that $t>a$ there exists $\dl^{'}\in(0,t-a)$ such that
\be \label{reslem}
\bigg|\int_{-\infty}^{a}[(t+\dl-r)^{d-1}-(t-r)^{d-1}] L(r) dr \bigg|  \leq C \cdot |\dl|^d, \ \ \forall |\dl|\leq \dl^{'},
\ee
where $C$ is a constant that may depend on $\omega,t,\dl^{'}$.
\end{lemma}

\begin{proof}
Fix an arbitrary $t\in \mathds{R}$ and pick $\dl^{'}\in(0,t-a)$.
For all $|\dl|\leq \dl^{'}$
\begin{eqnarray} \label{I2}
 \int_{-\infty}^{a}[(t+\dl-r)^{d-1}-(t-r)^{d-1}]L(r)dr
 &=& \int_{-a}^{N}[(t+\dl+u)^{d-1}-(t+u)^{d-1}]L_2(u)du \nonumber \\ \nonumber \\
 &&{} + \int_{N}^{\infty}[(t+\dl+u)^{d-1}-(t+u)^{d-1}]L_2(u)du  \nonumber \\ \nonumber \\
 &=:& I_{2,1}(N,\dl)+I_{2,2}(N,\dl).
\end{eqnarray}
Now we use the result on the long time behavior of L\'{e}vy processes. By Proposition 48.9 from \cite{sato}, if $E(L_2(1))=0$ and $E(L_2(1)^2)\leq \infty$, then
\be \label{star}
\limsup_{s \rr \infty} \frac{L_2(s)}{(2s \log \log(s))^{1/2}}=(E[L_2(1)^2])^{1/2}, \ \ P-\rm{a.s.}
\ee
Recall that $d\leq 0.5$. Hence by (\ref{star}) we can pick $N=N(\omega)>0$ large enough such that
\begin{eqnarray} \label{I22}
|I_{2,2}(N,\dl)| & \leq & \int_{N}^{\infty}|(t+\dl+u)^{d-1}-(t+u)^{d-1}| u^{1/2+\eps} du   \nonumber  \\
&\leq& C \cdot |\dl|, \ \ \forall \dl \in (-\dl^{'},\dl^{'}), \ \ P-\rm{a.s.}
 \end{eqnarray}
On the other hand, for $\dl$ small enough
\begin{eqnarray} \label{I21}
 |I_{2,1}(N,\dl)| &=& \bigg| \int_{-a}^{N}[(t+\dl+u)^{d-1}-(t+u)^{d-1}]L_2(u)du \bigg|  \\ \nonumber \\
 & \leq & C||L_2(u)||_{[0,N]} |\dl|\cdot (t-a)^{d-1}, \ \ \forall \dl \in (-\dl^{'},\dl^{'}),   \nonumber
 \end{eqnarray}
where
\bd
||L_2(u)||_{[0,N]}=\sup_{u\in[0,N]}|L_2(u)|.
\ed
Then, by (\ref{I22}) and (\ref{I21}) we get for $d<1/2$
\begin{eqnarray} \label{I2221}
|I_{2,1}(N,\dl)+I_{2,2}(N,\dl)| &<& C|\dl|, \ \ \forall \dl \in (-\dl^{'},\dl^{'}),
\end{eqnarray}
and by combining (\ref{I2}) with (\ref{I2221}) the result follows.
\end{proof}

\paragraph{Proof of Theorem \ref{fracM}}
By Theorem 3.4 in \cite{A4} we have
\be \label{parts}
M_d(t)=\frac{1}{\Gamma(d)}\int_{-\infty}^{\infty}[(t-r)^{d-1}_{+}-(-r)^{d-1}_{+}]L(r) dr , \   \ t \in \mathds{R},\ \ P-\rm{a.s.}
\ee
We prove the theorem for the case of $t > 0$. The proof for the case of $t\leq 0$ can be easily adjusted along the similar lines. We can decompose $M_d(t)$ as follows:
\begin{eqnarray*}
M_d(t)&=&\frac{1}{\Gamma(d)}\int_{0}^{t}(t-r)^{d-1}L(r)dr + \frac{1}{\Gamma(d)}\int_{-\infty}^{0}[(t-r)^{d-1} -(-r)^{d-1}]L(r)dr \\\\
&=& M^{1}_d(t) +M^{2}_d(t),  \   \ t \in (0,1),\ \ P-\rm{a.s.}
\end{eqnarray*}
By Lemma \ref{IntParts} we have
\bd
M^{1}_d(t)=\frac{1}{\Gamma(d+1)}\int_{0}^{t}(t-r)^{d} dL_r, \   \ t \in \mathds{R}_{+},\ \ P-\rm{a.s.}
\ed
By Theorem \ref{thm2} we have
\begin{displaymath}
\lim_{h\downarrow 0} \ \  \sup_{0< s<t <1,\ \ |t-s|\leq h}\Gamma(d+1) \frac{|M^{1}_d(t)-M^{1}_d(s)|}{h^d} =  \sup_{s\in [0,1]} |\Delta _{X}(s)| , \ \ P-\rm{a.s.}
\end{displaymath}
Therefore, $P$-a.s. $\omega$, for any $t\in(0,1)$, there exists $\dl_{1}>0$ and $C_1>0$ such that
\be \label{M1}
|M^{1}_d(t+\dl)-M^{1}_d(t)|\leq C_1|\dl|^d, \ \ \forall \dl\in (-\dl_1,\dl_1).
\ee
By Lemma \ref{fracL}, $P$-a.s. $\omega$, for any $t\in(0,1)$, there exists $\dl_{2}>0$ and
$C_2=C_2(\omega,t)>0$ such that
\be \label{M2}
|M^{2}_d(t+\dl)-M^{2}_d(t)|\leq C_2|\dl|^d, \ \ \forall \dl\in (-\dl_2,\dl_2).
\ee
Hence by (\ref{M1}) and (\ref{M2}), $P$-a.s. $\omega$, for any $t\in(0,1)$, we can fix $\dl_3$ and $C=C(\omega,t)$ such that,
\bd
|M_d(t+\dl)-M_d(t)| \leq C|\dl|^d,  \ \ \forall \dl\in (-\dl_3,\dl_3),
\ed
and we are done.
\qed \\

\bibliographystyle{plain}
\printindex

\end{document}